\documentclass{amsart}
\title[Metric spaces quasi-isometric to trees]{Real valued functions and metric spaces quasi-isometric to trees.}
\date{}

\author[Á.~Martínez-Pérez]{Álvaro Martínez-Pérez}
\address{Departamento de Geometría y Topología\\ Universidad Complutense de Madrid\\ Madrid 28040  Spain}
\email{alvaro\_martinez@mat.ucm.es}

\thanks{Departamento de Geometría y Topología\\ Universidad Complutense de Madrid\\ Madrid 28040  Spain\\email: alvaro\_martinez@mat.ucm.es}

\usepackage{amsfonts}
\usepackage{amsmath}
\usepackage{latexsym}
\usepackage[all]{xy}
\usepackage{amsthm}
\usepackage{graphicx}
\usepackage[T1]{fontenc}
\usepackage{hyperref}
\usepackage{caption}
\usepackage{verbatim}
\usepackage{amsmath}
\usepackage{amssymb}
\usepackage{amsthm}
\usepackage{amscd}
\usepackage{graphics}
\usepackage{graphpap}
\usepackage{float}

\newtheorem{definicion}{Definition}[section]
\newtheorem{nota}[definicion]{Remark}
\newtheorem{prop}[definicion]{Proposition}
\newtheorem{lema}[definicion]{Lemma}

\newtheorem{teorema}[definicion]{Theorem}
\newtheorem{cor}[definicion]{Corollary}

\newcommand{\co}{\ensuremath{\colon}} 
\newcommand{\bz}{\ensuremath{\mathbb{Z}}} 
\newcommand{\br}{\ensuremath{\mathbb{R}}} 

\begin{document}

\begin{abstract} We prove that if $X$ is a complete geodesic metric space with uniformly generated $H_1$ and $f\co X\to \br$ is metrically proper on the connected components and bornologous, then $X$ is quasi-isometric to a tree.

Using this and adapting the definition of hyperbolic approximation we obtain an intrinsic sufficent condition for a metric space to be PQ-symmetric to an ultrametric space.
\end{abstract}

\maketitle

\tableofcontents

\begin{footnotesize}
\end{footnotesize}

\begin{footnotesize}
\end{footnotesize}

\section{Introduction}

A geodesic metric space $X$ is a path-connected metric space in which any two points $x,y$ are connected by an isometric image of an interval in the real line, called a geodesic and denoted by $[x,y]$. A geodesic metric space $X$ is called Gromov hyperbolic if there exists some $\delta\geq 0$ such that for any geodesic triangle $[xy]\cup[yz]\cup [zx]$ in $X$ each side is contained in a $\delta$-neighborhood of the other two. Several equivalent definitions can be found in \cite{B-H}. By a tree, we refer to a 1-dimensional simply connected simplicial complex. This is an example of Gromov 0-hyperbolic space. All metric spaces in this paper are assumed to be unbounded.

A map between metric spaces, $f:(X,d_X)\to (Y,d_Y)$, is
said to be \emph{quasi-isometric} if there are constants $\lambda
\geq 1$ and $C>0$ such that $\forall x,x'\in X$,
$\frac{1}{\lambda}d_X(x,x') -C \leq d_Y(f(x),f(x'))\leq \lambda
d_X(x,x')+C$. If if there is a constant $D>0$ such that
$\forall y\in Y$, $d(y,f(X))\leq D$, then $f$ is a
\emph{quasi-isometry} and $X,Y$ are \emph{quasi-isometric}.

In the case $\lambda=1$, the map $f$ is called \emph{roughly isometric} and a \emph{rough isometry} respectively.

Note that the composition of quasi-isometries is also a quasi-isometry.

There are several results in the literature characterizing when a metric space, or a graph, is quasi-isometric to a tree. In section \ref{Section: carac} we compile some of them from the perspective of the asymtotic dimension of the space, see \cite{D1} and \cite{F-W}, or an intrinsic property on the geodesics, see \cite{Man}. One reason to be interested in this is that any quasi-action on a geodesic metric space $X$ is quasi-conjugate to an action on some connected graph quasi-isometric to $X$. The converse is also true. In particular, any quasi-action on a simplicial tree is quasi-conjugate to an isometric action on a quasi-tree and any isometric action on a quasi-tree is quasi-conjugate to a quasi-action on a simplicial tree. In the case of bounded valence bushy trees it is proved in \cite{MSW} that any cobounded quasi-action is quasi-conjugated to an action on another bounded valence bushy tree.

The first aim in this work is to provide a new sufficent condition 
for a space to be quasi-isometric to a tree. To do this, we were inspired by the 
work of Manning in \cite{Man} where he studies the geometry of pseudocharacters, this is, real valued functions on groups which are ``almost'' homeomorphisms. In his construction, he uses a pseudocharacter with some conditions on it to obtain a tree from the Cayley graph of a finitely presented group.

We consider a real valued function $f$ on a graph with uniformly generated $H_1$ and do a similar thing. Let us recall that $H_1$ is \emph{uniformly generated} if there is an $L > 0$ so that $H_1(X)$ is generated by loops of length at most $L$. We assume $f$ to be bornologous and introduce the definition of \emph{metrically proper on the connected components}, see section \ref{section 2} for definitions. Then we extend the result for complete geodesic spaces.

\begin{teorema}\label{Tma: intro} Let $X$ be a complete geodesic space with $H_1(X)$ uniformly generated. If there is a function $f\co X \to \br$ such that $f$ is bornologous and metrically proper on the connected components, then $X$ is quasi-isometric to a 
tree. 
\end{teorema}

An \emph{ultrametric space} is a metric space $(X,d)$ such that 
$$d(x,y)\leq \max \{d(x,z),d(z,y)\}$$
for all $x,y,z\in X$.

A map $f:X \to Y$ between metric spaces is
called \emph{quasi-symmetric} if it is not constant and if there
is a homeomorphism $\eta:[0,\infty) \to [0,\infty)$ such that from
$|xa|\leq t|xb|$ it follows that $|f(x)f(a)|\leq
\eta(t)|f(x)f(b)|$ for any $a,b,x\in X$ and all $t\geq 0$. 
The function $\eta$ is called the \emph{control function} of $f$.

A quasi-symmetric map is said to be \emph{power quasi-symmetric}
or \emph{PQ-symmetric}, if its control function is of the form
\[\eta(t)= q \max\{t^p,t^{\frac{1}{p}}\}\] for some $p,q\geq 1$.


In \cite{BS}, S. Buyalo and V. Schroeder introduce a special kind
of hyperbolic cones called hyperbolic approximations, which are
defined in general for non-necessarily bounded metric spaces. See section \ref{Section: Hyp ap} for definitions. This hyperbolic approximations $\mathcal{H}$ have a canonical level function which is a real valued map and, since $\mathcal{H}$ is a Gromov hyperbolic space it has uniformly generated $H_1$. Thus, Theorem \ref{Tma: intro} naturally yields a sufficent contition on the hyperbolic approximation to be quasi-isometric to a tree. 

In \cite{M} we prove that two metric spaces are PQ-symmetric if and only if their hyperbolic approximations are quasi-isometric. Also, it is immediate to see that the hyperbolic approximation of an ultrametric space is a tree. Therefore, we can use the condition above to conclude that a certain metric space is PQ-symmetric to an ultrametric space. In section \ref{sec: PQ} we introduce an alternative construction to the hyperbolic approximation. This allows us to solve some technical problems and find the following intrinsic sufficent condition for 
a metric space to be PQ-symmetric to an ultrametric space.

A $\varepsilon$-\emph{chain} is a finite sequence of points $x_0, ..., x_N$ that are separated by distances of $\varepsilon$ or less: $|x_i - x_{i+1}| < \varepsilon$. Two points are $\varepsilon$-\emph{connected} if there is a $\varepsilon$-chain joining them. Any two points in a $\varepsilon$-\emph{connected set} can be linked by a $\varepsilon$-chain. A $\varepsilon$-\emph{connected component} is a maximal $\varepsilon$-connected subset. 

\begin{definicion} A metric space $X$ is $D$-\textbf{finitely} $\varepsilon$-\textbf{connected} if for any two points $x,x'\in X$ there is a $\varepsilon$-chain $x=x_0,x_1,...x_N=x'$ with $N\leq D$.
\end{definicion}

\begin{cor} Let $Z$ be a metric space. If there are constants $D>0$ and $0<r<\frac{1}{6}$ such that every $r^k$-connected component is $D$-finitely $r^k$-connected for any $k\in \bz$, then $Z$ is PQ-symmetric to an ultrametric space.
\end{cor}

\section{Real valued functions on metric spaces.}\label{section 2}

\begin{definicion} Given a map $f\co X \to Y$ between metric spaces, a non-decreasing function $\varrho_f\co J\to [0,\infty)$ with $J=[0,T]$ or $J=[0,\infty)$ is called \emph{expansion function} if $\forall A \in X$ with $diam(A) \in J$, \
$diam(f(A))\leq \varrho_f(diam(A))$.
\end{definicion}

A map $f: X_1 \to X_2$ between two metric
spaces is \emph{bornologous} if for every $R>0$ there is $S>0$
such that for any two points $x,x'\in X_1$ with $d(x,x')<R$,
$d(f(x),f(x'))<S$. For convinience, we are going to use also the following equivalent definition.

\begin{definicion} A map $f\co X \to Y$ between metric spaces is called 
\emph{bornologous} if there is an expansion function $\varrho_f \co [0,\infty) \to [0,\infty)$.
\end{definicion}


\begin{definicion}
A map $f$ between two metric spaces $X,  X'$ is \emph{metrically
proper} if for any bounded set $A$ in $X'$, \ $f^{-1}(A)$ is
bounded in $X$.
\end{definicion}

A map is called \emph{coarse} if it is metrically proper
and bornologous. Two metric spaces $X,Y$ are \emph{coarse equivalent} if there are maps $f\co X \to Y$ and $g\co Y \to X$ such that $f\circ g$ and $g\circ f$ are close to the identity. Although this notion is more general than that of quasi-isometry it is well known, and an easy exercise, that if $X,Y$ are geodesic spaces, then $X,Y$ are coarse equivalent if and only if they are quasi-isometric. See \cite{Roe1}. In this work, considering the references, talking about quasi-isometry seems more natural. However, for the proof of the following theorem coarse approach will be useful.

\begin{definicion} Given a metric space $X$, we say that $f\co X \to \br$ is \textbf{metrically proper on the connected components} if $\forall \, N>0$ there is some $M>0$ such that for any interval $[x-N,x+N]$, the diameter of every connected component of $f^{-1}[x-N,x+N]$ is bounded above by $M$.
\end{definicion}

\begin{teorema}\label{Th: qi} Let $X$ be a graph with $H_1(X)$ uniformly generated and $|X|$ be the geometric realization where every edge has length 1. If there is a function
$f\co |X| \to \br$ such that $f$ is bornologous and metrically proper on the connected components, then $|X|$ is quasi-isometric to a 
tree. In particular, $|X|$ is Gromov hyperbolic.
\end{teorema}

\begin{proof} Let us rescale $f$ to ensure that $\varrho_f(L)<\frac{1}{4}$ 
where $L\in \mathbb{Z}_+$ is an upper bound for the length of the loops generating $H_1(X)$. Now, let us build a simply connected 2-complex $Y$ quasi-isometric to $X$ following the idea in \cite{F-W}. Let $A$ be a maximal collection of vertices in $X$ with $d(a,a')\geq L$ for all $a\neq a'$. 
Let $R=3L$, and let $Y$ be the space $$X\cup_{a\in A}cone (\bar{B}(a,R))$$

In words, $Y$ is $X$ with each closed $R$\--ball centered at $a \in A$ coned to a point. Give $Y$
the induced path metric where each cone line has length $R$. The inclusion of $X$ is then
isometric, and has coarsely dense image, and so is a quasi-isometry. Also, the resulting space $Y$ is simply connected (see \cite{F-W}). Let us denote by $v_a$ the cone vertex for the ball $B(a,R)$.

Let us define a map \begin{equation}\label{Extension FW} \tilde{f}\co Y\to \br \end{equation} as a coarse extension of $f$. For every vertex $x\in X\subset Y$, let $\tilde{f}(x)\in (f(x)-\frac{1}{4},f(x)-\frac{1}{4})\backslash \bz$. For every $v_a$, let $\tilde{f}(v_a)=\tilde{f}(a)$. Let $\tilde{f}$ be the affine extension on $Y$.

Let us show, just by triangle inequality, that $\tilde{f}|_X$ is $1$-close to $f$. For any vertex $x_i$, $|f(x_i)-\tilde{f}(x_i)|<\frac{1}{4}$. Let $x$ be a point in the realization of an edge $[x_1,x_2]$ and let us assume that $d(x,x_1)\leq \frac{1}{2}$. Since $\varrho_f(L)<\frac{1}{4}$ with $L\geq 1$, $|f(x)-f(x_i)|<\frac{1}{4}$ for $i=1,2$. Hence, $|\tilde{f}(x)-\tilde{f}(x_1)|\leq \frac{1}{2}|\tilde{f}(x_2)-\tilde{f}(x_1)|<\frac{1}{2}\frac{3}{4}$ and we conclude that $|f(x)-\tilde{f}(x)|<1$.

$\tilde{f}$ is bornologous. In particular, the image of every simplex has diameter bounded above by $R\frac{3}{4}$.

Let $\mathcal{V}$ be the vertex set of $Y$. Since $\tilde{f}^{-1}(k)\cap V=\emptyset \ \forall k \in \bz$ then for any simplex $\Delta \in Y$, and every $t\in \br$, $\tilde{f}^{-1}(t)\cap \Delta$ is a track.

Thus, $\tau=\tilde{f}^{-1}(\bz)$ is a union of tracks in $Y$. Each such a track separates $Y$ in two connected components and has a product neighborhood $\eta(\tau)=\tau \times (-\frac{1}{2},\frac{1}{2})$ in the complement of the 0-skeleton $\mathcal{V}$ of $Y$. A quotient space $T$ of $Y$ is obtained by smashing each component of $\eta(\tau)$ to an interval and each component of the complement of $\eta(\tau)$ to a point.


Let $\pi\co Y\to T$ be the quotient map. Clearly $T$ is a simplicial graph. Since $Y$ is simply connected and the preimage of each point under $\pi$ is connected, $T$ must be simply connected. In particular, $T$ is a tree.

Claim. The quotient map is a quasi-isometry. Since both spaces are geodesic, as explained above, it suffices to check that it is a coarse equivalence. Let $K>0$ and consider $x,y\in Y$ with $d(x,y)<K$. 

Since the 2-complex is quasi-isometric to its 1-skeleton, there are constants $\lambda,C>0$ and a sequence of vertices $x_1,...x_{k}$ such that $x,y$ are in simplices adjacent to $x_1,x_k$ respectively, $\{x_i,x_i+1\}$ is joined by an edge and $k\leq \lambda K+C$.

Clearly, either $\pi(x_i)$ and $\pi(x_i+1)$ are the same vertex in $T$ or else, there is a track in $\tau$ which crosses the edge $\{x_i,x_{i+1}\}$ which implies that there is an edge between $\pi(x_i)$ and $\pi(x_i+1)$. Therefore, it is trivial to check that $d(\pi(x),\pi(y))\leq k+2\leq \lambda K+ C+2$ proving that $\pi$ is bornologous.

To check that $\pi$ is metrically proper it suffices to see that for every 1-simplex $e\in T$, $\pi^{-1}(e)$ has uniformly bounded diameter which is immediate since $f$ is metrically proper on the connected components.

As a coarse inverse of $\pi$, let us define a map $i\co T\to Y$ such that for any $w\in \mathcal{W}$, the vertex set of $T$, $i(w)$ is any point in the corresponding vertex set of $Y$ and for any $x\in T\backslash \mathcal{W}$, if $x\in e\in \mathcal{E}(T)$, $i(x)$ is any point in the corresponding component of $\tau$. It is trivial to check that $i$ is a coarse inverse for $\pi$.
\end{proof}

The aim of the rest of the section is to extend 
this result to complete geodesic spaces.

A subset $A$ in a metric space $X$ is called \emph{R-separated},
$R>0$, if $d(a,a')\geq R$ for any distinct $a,a'\in A$. Note that
if $A$ is maximal with this property, then the union $\cup_{a\in
A} B_R(a)$ covers $X$.

Fix a constant $R>0$ and let $A$ be a maximal
$R$-separated set. Let us define a graph
$\Gamma(X,R,A)$ as follows. For every $a\in A$, consider the ball
$B(a,2R)\subset X$. Let $V$ be the the set of balls $B(a,2R)$, $a\in A$. Therefore, if
for some $a,a'\in A$, $B(a,2R)=B(a',2R)$, then they represent the same point $v$
in $V$. Let us denote the corresponding ball simply by $B(v)$. Let $V$ be the vertex set of $\Gamma(X,R,A)$.
Vertices $v,v'$ are connected by an edge if and only the close balls
$\bar{B}(v),\bar{B}(v')$ intersect, $\bar{B}(v)\cap
\bar{B}(v')\neq \emptyset$.

Consider the path metric on the geometric realization $|\Gamma(X,R,A)|$ for which every edge has length 1. $|vv'|$
denotes the distance between points $v,v'\in V$ in $|\Gamma(X,R,A)|$, while
$d(a,a')$ denotes the distance in $X$. 

Let us define a map $j\co |\Gamma(X,R,A)| \to X$ such that for any vertex $v\in V$, $j(v)=a$ for some $a\in A$ with $B(a,2R)=B(v)$ and for any edge with realization $e=[v,v'] \in |\Gamma(X,R,A)|$, let $j\co e \to X$ be a geodesic path $[j(v),j(v')]$.

\begin{prop}\label{Prop: coarse approx} $j\co |\Gamma(X,R,A)| \to X$ is a quasi-isometry.
\end{prop}

\begin{proof} First, let us consider the restriction to the vertices in $\Gamma(X,R,A)$. Suppose $v,v'\in V$ with $|vv'|=k$. Then, there are vertices $v_0=v,v_1,...,v_{k-1},v_k=v'$ such that $\{v_{i-1},v_i\}$ is an edge in $\Gamma(X,R,A)$, i.e. $\bar{B}(v_{i-1})\cap \bar{B}(v_{i})\neq \emptyset$ and, therefore, 
\begin{equation}\label{Eq 1} d(j(v),j(v'))\leq 4R|vv'|. \end{equation}

Let $\gamma$ be a geodesic path $[j(v),j(v')]$ of length $l$. Let $k'=[\frac{l}{R}]+1$ and $x_i$ the point $\gamma(\frac{i\cdot l}{k'})$ for $i=0,k'$. Since $A$ is an $R$-separated set, there is some $a_i\in B(x_i,R)$ for every $i=1,k'-1$ and, considering $v_0=v$, $v_{k'}=v$ it is immediate to check that $\bar{B}(v_{i-1})\cap \bar{B}(v_{i})\neq \emptyset$ for $i=1,k'$. This implies that $|vv'|\leq k'\leq \frac{l}{R} +1$ and thus
\begin{equation}\label{Eq 2} R|vv'|-R \leq d(j(v),j(v')).\end{equation}

If $x,x'\in \Gamma(X,R,A) \backslash V$, the upper bound for $d(j(x),j(x'))$ is trivially $4R|xx'|$. For the lower bound consider a geodesic path in $[x,x']$ in $\Gamma(X,R,A)$ and let $v,v'$ the vertices in that path adjacent to $x$ and $x'$ respectively. Then, $d(j(x),j(v))\leq 4R$ and $d(j(x'),j(v'))\leq 4R$ by construction. From triangle inequality and equation (\ref{Eq 2}) we finally obtain that
\begin{equation}\label{Eq 3} R|x,x'|-9R-2 \leq d(j(x),j(x'))\leq 4R|x,x'|.\end{equation}

It is trivial from the construction that any point of $X$ is at most at distance $2R$ from $j(\Gamma(X,R,A))$. Thus, $j$ is a $(4R,9R+2)$-quasi-isometry.
\end{proof}

\begin{lema}\label{Lemma 2.3 FW} \cite[Lemma 2.3]{F-W} Let $X$ be a complete geodesic metric space. The following are equivalent:
\begin{itemize}\item $X$ has uniformly generated $H_1$.
\item $X$ is quasi-isometric to a complete geodesic metric space, $Y$, with $H_1(Y)=0$.
\end{itemize}
\end{lema}

(These are also equivalent to the condition $H^{uf}_1 (X)= 0$ which is the form which
appears in \cite{BW})

\begin{prop}\label{Prop: coarse approx2} If $X$ is a complete geodesic space with uniformly generated $H_1$, then $|\Gamma(X,R,A)|$ is quasi-isometric to $X$ and $H_1(\Gamma(X,R,A))$ is uniformly generated.
\end{prop}

\begin{proof} By Proposition \ref{Prop: coarse approx}, $|\Gamma(X,R,A)|$ is quasi-isometric to $X$. 

By Lemma \ref{Lemma 2.3 FW}, $X$ is quasi-isometric to a complete geodesic metric space, $Y$, with $H_1(Y)=0$. Therefore, again by \ref{Lemma 2.3 FW}, $H_1(\Gamma(X,R,A))$ is uniformly generated.
\end{proof}

Given a function $f\co X \to \br$, let us define $\hat{f}\co |\Gamma(X,R,A)| \to \br$ such that for any vertex $v\in V$, $\hat{f}(v)=f(j(v))$. Then extend $\hat{f}$ affinely on the edges.

\begin{prop}\label{Prop: approx function} Let $X$ be a complete geodesic space. 
If $f\co X \to \br$ is bornologous and metrically proper on the connected components then for any $R>0$ and any maximal $R$--separated set $A\subset X$, $\hat{f}\co |\Gamma(X,R,A)| \to \br$ holds the same properties.
\end{prop}

\begin{proof} It is readily seen that $\varrho_{\hat{f}}(1)\leq \varrho_f(2R)$. Since $|\Gamma(X,R,A)|$ is a geodesic space, this proves that $\hat{f}$ is bornologous.

Consider now any connected component $\mathcal{C}$ of $\hat{f}^{-1}(x-N,x+N))$. For any edge $\{v_1,v_2\}$ contained in $\mathcal{C}$ if $j(v_i)=a_i\in A$ for $i=1,2$, then $d(a_1,a_2)\leq 4R$. Since $X$ is geodesic and bornologous, $a_1,a_2$ are in the same connected component $\mathcal{D}$ of $f^{-1}(x-N-\varrho_f(4R),x+N+\varrho_f(4R))$. Then, all the vertices in $\mathcal{C}$ are contained in $\mathcal{D}$. The diameter of $\mathcal{D}$ is uniformly bounded (independently of $x$) by assumption on $f$. Therefore, since $j\co |\Gamma(X,R,A)| \to X$ is a quasi-isometry, $\hat{f}$ is metrically proper on the connected components.
\end{proof}

Next theorem follows immediately from Theorem \ref{Th: qi} together with propositions \ref{Prop: coarse approx2} and \ref{Prop: approx function}.

\begin{teorema}\label{Th: qi2} Let $X$ be a complete geodesic space with $H_1(X)$ uniformly generated. If there is a function $f\co X \to \br$ such that $f$ is bornologous and metrically proper on the connected components, then $X$ is quasi-isometric to a tree. In particular, $X$ is Gromov hyperbolic. 
\end{teorema}

If a metric space is Gromov hyperbolic then $H_1(X)$ is uniformly generated. See, for example, the proof of Corollary 1.5 in \cite{F-W}.

\begin{cor}\label{Cor: qi2} Let $X$ be a complete Gromov hyperbolic space. If
$f\co X \to \br$ is a bornologous function and $f$ is metrically proper on the connected components, then $X$ is quasi-isometric to a tree.
\end{cor}

\section{Characterizations for a metric space to be quasi-isometric to a tree.}\label{Section: carac}

\begin{teorema}\cite[Theorem 4.6]{Man} Let $X$ be a geodesic metric space. The following are equivalent:
\begin{itemize}\item[(1)] $X$ is quasi-isometric to some simplicial tree $\Gamma$.
\item[(2)] (Bottleneck Property) There is some $\Delta>0$ so that for all $x,y \in Y$ there is a midpoint $m=m(x,y)$ with $d(x,m)=d(y,m)=\frac{1}{2}d(x,y)$ and the property that any path from $x$ to $y$ must pass within less than $\Delta$ of the point $m$.
\end{itemize}
\end{teorema}

If $X$ is a set and $X=\cup_i O_i$ a covering, the \emph{multiplicity} of the covering is at most $n$ if any point $x\in A$ is contained in at most $n$ elements of $\{O_i\}$. For $D>0$, the $D$--multiplicity of the covering is at most $n$ if for any $x\in X$, the closed $D$--ball intersects at most $n$ elements of $\{O_i\}$. The \emph{asymptotic dimension} of the metric space $X$ is at most $n$ if for any $D\geq 0$ there exist a covering $X=\cup_i O_i$ such that the diameter of $O_i$ is uniformly bounded (i.e. there is some $C>0$ such that $diam(O_i)\leq C$ for every $i$) and the $D$--multiplicity of the covering is at most $n+1$. The asymptotic dimension of $X$ is $n$, $asdim(X)=n$, if the asymptotic dimension is at most $n$, but it is not at most $n-1$.

The following is due to M. Cendelj et al. in \cite{D1}. In that work they present a combinatorial approach to coarse geometry unsing direct sequences. The spirit is pretty much the same of the inverse system approach to shape theory (see \cite{Dydak-Segal} and \cite{MS1}).

Let $X$ be a metric space and 
$R \in \br_+$. Then the $R$-Rips complex Rips$_R(X)$ is the
simplicial complex whose vertices are points of $X$; vertices $x_1,..., x_n$ span a simplex iff $d(x_i, x_j)\leq R$
for each $i,j$.

For each pair $0 \leq r \leq R < \infty$ there is a natural simplicial embedding $$\iota_{r,R} \co \mbox{Rips}_r(X) \to \mbox{Rips}_R(X),$$ so that $\iota_{r,\rho}=\iota_{R,\rho} \circ \iota_{r,R}$ provided that $r \leq R \leq \rho$.

\begin{definicion}\cite[Definition 2.10]{K} A metric space $X$ is \emph{coarsely} $k$-\emph{connected} if for each $r$ there exist $R>r$ so that the mapping $|Rips_r(X)|\to |Rips_R(X)|$ induces a trivial map of $\pi_i$ for $0\leq i \leq k$.
\end{definicion}

A metric space $X$ is coarsely homology $n$-\emph{connected} if for each $r$ there exist $R>r$ so that the mapping $|Rips_r(X)|\to |Rips_R(X)|$ induces a trivial map of reduced homology groups $\tilde{H}_i$ for $0\leq i \leq k$.

\begin{teorema}\cite[Theorem 7.1]{D1} If $X$ is a coarsely geodesic metric space, then the following conditions are equivalent:
\begin{itemize}\item[(1)] $X$ is coarsely equivalent to a simplicial tree,
\item[(2)] $asdim(X)\leq 1$ and $X$ is coarsely homology 1-connected,  
\item[(3)] $asdim(X)\leq 1$ and $X$ is coarsely 1-connected.  
\end{itemize}
\end{teorema}

Which implies the following, 

\begin{teorema}\cite[Theorem 1.1]{F-W} Let $X$ be a geodesic metric space with $H_1(X)$ uniformly generated. If
$X$ has asymptotic dimension one then $X$ is quasi-isometric to an unbounded tree.
\end{teorema}

Now, by Theorem \ref{Th: qi2},

\begin{cor}\label{cor: qi} Let $X$ be a complete geodesic space with uniformly generated $H_1$. If $f\co X \to \br$ is bornologous and metrically proper on the connected components, then $X$ holds the following equivalent conditions:
\begin{itemize}\item[(1)] $X$ is quasi-isometric to a tree,
\item[(2)] $X$ has bottleneck property,
\item[(3)] $asdim(X)\leq 1$ and $X$ is coarsely homology 1-connected,  
\item[(4)] $asdim(X)\leq 1$ and $X$ is coarsely 1-connected,
\item[(5)] $asdim(X)\leq 1$ and $H_1(X)$ is uniformly generated,
\end{itemize}
\end{cor}

\section{Hyperbolic approximation.}\label{Section: Hyp ap}

Let us recall here the construction of the hyperbolic approximation
introduced in \cite{BS}.

A \emph{hyperbolic approximation} of a metric space $Z$ is a graph
$X=\mathcal{H}(Z)$ which is defined as follows. Fix a positive $r\leq
\frac{1}{6}$ which is called the \emph{parameter} of $X$. For
every $k\in \mathbb{Z}$, let $A_k\in Z$ be a maximal
$r^k$-separated set. For every $v\in A_k$, consider the ball
$B(v)\subset Z$ of radius $r(v):=2r^{k}$ centered at $v$. Let $V_k$ be the 
set of balls $B(v)$, $v\in A_k$ and $V$ the union, for $k\in
\mathbb{Z}$, of $V_k$. Therefore, if
for any $v,v'\in A_k$, $B(v)=B(v')$, they represent the same point
in $V$, but if $B(v_k)=B(v_{k'})$ with $k\neq k'$, then they yield
different points in $V$. Let $V$ be the vertex set of a graph $X$.
Vertices $v,v'$ are connected by an edge if and only if they
either belong to the same level, $V_k$, and the close balls
$\bar{B}(v),\bar{B}(v')$ intersect, $\bar{B}(v)\cap
\bar{B}(v')\neq \emptyset$, or they lie on neighboring levels
$V_k,V_k+1$ and the ball of the upper level, $V_{k+1}$, is
contained in the ball of the lower level, $V_k$.

An edge $vv'\subset X$ is called \emph{horizontal}, if its
vertices belong to the same level, $v,v'\in V_k$ for some $k\in
\mathbb{Z}$. Other edges are called \emph{radial}. Consider the
path metric on $X$ for which every edge has length 1. $|vv'|$
denotes the distance between points $v,v'\in V$ in $X$, while
$d(v,v')$ denotes the distance between them in $Z$. There is a
natural level function $l:V \to \mathbb{Z}$ defined by $l(v)=k$
for $v\in V_k$. Consider also the canonical extension $l:X \to
\mathbb{R}$.






A hyperbolic approximation of any metric space is a Gromov hyperbolic space. As we mentioned at the end of section \ref{section 2}, this implies that the first homology group is uniformly generated.

\begin{nota}\label{nota: tree} If $Z$ is an ultrametric space then if two balls intersect, one is contained in the other. Therefore, there are no horizontal edges and $\mathcal{H}(X)$ is a tree.
\end{nota}

It is well known the correspondence between trees and ultrametric spaces. By choosing a root on an $\br$--tree the boundary at infinity naturally becomes a complete ultrametric space. In fact, several categorical equivalences are proved in the literature, see \cite{Hug} and \cite{M-M}.

It is also known that a quasi-isometry between Gromov hyperbolic spaces induces 
between their ultrametric end spaces a homeomorphism with the property of being power quasi-symmetric, or PQ-symmetric. See \cite[Theorem 5.2.15]{BS}. As a converse of this, let us recall the following.

\begin{teorema}\label{tma: M}\cite[Theorem 4.14]{M} For any PQ-symmetric homeomorphism $f:Z\to Z'$ of complete
metric spaces, there is a quasi-isometry of their hyperbolic
approximations $F: X \to X'$ which induces $f$, $\partial_\infty
F(z)=f(z) \, \forall z\in Z$. 
\end{teorema}

Then, by Remark \ref{nota: tree},

\begin{cor}\label{cor: M} If $f:Z\to Z'$ is a PQ-symmetric homeomorphism of complete metric spaces and $Z'$ is an ultrametric space then $\mathcal{H}(Z)$ is quasi-isometric to a tree.
\end{cor}

Tukia and V{\"a}is{\"a}l{\"a} \cite{TukVais}
proved that a quasi-symmetric homeomorphism between uniformly perfect metric spaces is PQ-symmetric (see also
\cite[Theorem 11.3, page 89]{Hei}). This, toguether with \ref{cor: qi}, \ref{tma: M} and \ref{cor: M} yields the following.

\begin{teorema}\label{Th: application} If $Z$ is a complete metric space and $X=\mathcal{H}(Z)$ is its hyperbolic approximation, the following are equivalent:
\begin{itemize}\item[(1)] $X$ is quasi-isometric to a tree,
\item[(2)] $X$ has bottleneck property,
\item[(3)] $asdim(X)\leq 1$,
\item[(3)] $Z$ is PQ-symmetric homeomorphic to a bounded complete ultrametric space.
\end{itemize}

Moreover, if $Z$ is uniformly perfect, all of these are equivalent to

\begin{itemize}
\item[(4)] $Z$ is quasi-symmetric homeomorphic to a bounded complete ultrametric space.
\end{itemize}
\end{teorema}

The level function on a hyperbolic approximation is a natural example of a real valued function on a graph. Also, note that it is trivially bornologous.

\begin{cor} Let $Z$ be a metric space, $X=\mathcal{H}(Z)$ a hyperbolic approximation and $l\co X \to \br$ its level function. If $l$ is metrically proper on the connected components, then $Z$ is PQ-symmetric to an ultrametric space. Moreover, $X$ holds all the conditions in \ref{Th: application}.
\end{cor}

However, using the hyperbolic approximation as defined, it depends essentially on the election of the sets $A_k$ whether the level function is metrically proper on the connected components or not.

\section{PQ-symmetric homeomorphisms to ultrametric spaces.}\label{sec: PQ}

We can avoid the dependence on the election of the sets $A_k$ using an alternative definition of hyperbolic approximation. To do this, instead of using maximal $r^k$-separated sets and the intersections of the covering to produce a graph, we can use directly Rips graphs. 

Given a metric space $(Z,d_Z)$ and $t>0$, the \emph{Rips graph} RipsG$_t(Z)$ is a graph structure on $Z$ with edges $[x,y]$ such that $d_Z(x,y)\leq t$. 

Thus, let $0<r\leq \frac{1}{6}$ and $X_k:=$RipsG$_{r^k}(Z)$. Let $p_k\co Z \to X_k$ induced by the identity and let us denote $p_k(x)=x^k$. Let $X$ be the graph whose vertex set $W$ is the union of the vertices in $X_k$ for every $k$. The edges in $X_k$ for any $k\in \bz$ are called \emph{horizontal} edges in $X$. For $x_k\in X_k$ and $x_{k+1}\in X_{k+1}$ there is an edge $[x_k,x_{k+1}]$ if $d_X(x_k,x_{k+1})\leq r^k$. Let us denote by $X=\mathcal{RH}(Z)$ this alternative hyperbolic approximation. 

There is a
natural level function $l:W \to \mathbb{Z}$ defined by $l(x_k)=k$. Consider also the canonical extension $l:X \to
\mathbb{R}$.

To prove that $\mathcal{RH}(Z)$ is a Gromov hyperbolic space quasi isometric to $\mathcal{H}(Z)$ we shall need a few lemmas. We include the proofs for completion although some of the proofs are very similar to those in \cite{BS}.

\begin{lema}\label{lema: radial} For every $x, x'\in W$ there exists $y \in W$ with $l(w) \leq
l(v), l(v')$ such that $x, x'$ can be connected to $y$ by radial geodesics. In
particular, the space $X$ is geodesic.
\end{lema}

\begin{proof} Let $l(x)=k$ and $l(x')=k'$. Choose $m<\min\{k,k'\}$ small enough such that $d_X(x,x')\leq r^{m}$. Since $d_X(x,x')\leq r^{m}$, there is a radial edge $[x'^{m+1},x^m]$. Therefore,  $\gamma=x^k,x^{k-1},...,x^{m}$ and $\gamma'=x'^k,x'^{k-1},...,x'^{m+1},x^m$ are radial geodesics to a common point $y=x^m$. This implies that $X$ is geodesic because distances between vertices take integer values.
\end{proof}

Next lemma is immediate from the definition of the edges in $\mathcal{RH}(Z)$.

\begin{lema}\label{lema: escalon} Given an edge $[x,x']$ with $k=l(x)\geq l(x')=k'$ there is an edge $[w,x']$ for $w=p_{k-1}p_{k}^{-1}(x)$.
\end{lema}

\begin{lema} Any two vertices $x,x'\in W$ can be joined by a geodesic $\gamma=v_0...v_{n+1}$ with $v_0=x$, $v_{n+1}=x'$ such that $l(v_i) < \max\{l(v_{i-1}), l(v_{i+1})\}$ for all $1 \leq i \leq n$.
\end{lema}

\begin{proof} Let $n = |xx'| - 1$. Consider a geodesic $\gamma= v_0 . . . v_{n+1}$ from
$v_0 = x$ to $v_{n+1} = x'$ such that $\sigma(\gamma) = \sum_{i=1}^n  l(v_i)$ is minimal. Then let us see that $\gamma$ has the desired property.

Let $1 \leq i \leq n$, and let
$k = l(v_i)$. Consider the sequence $(l(v_{i-1}), l(v_i), l(v_{i+1}))$. There are
nine combinatorial possibilities for this sequence. To prove the result it
remains to show, that the sequences $(k-1, k, k-1)$, $(k, k, k)$, $(k-1, k, k)$
and $(k, k, k - 1)$ cannot occur and this follows immediately from Lemma \ref{lema: escalon} and the hypothesis that $\sigma$ is minimal.
\end{proof}

From this we easily obtain the following

\begin{lema}\label{lema: geodesicas} Any vertices $x, x'\in W$ can be connected in $X$ by a
geodesic which contains at most one horizontal edge. If there is such an
edge, then it lies on the lowest level of the geodesic.
\end{lema}

Let $W' \subset W$. A point
$y \in W$ is called a \emph{cone point for} $W'$ if $l(y) \leq  \inf_{x\in W'} l(x)$ and every $x \in W'$
is connected to $y$ by a radial geodesic. A cone point of maximal level
is called a \emph{branch point of} $W'$. By Lemma \ref{lema: radial}, for every two points $x,
x' \in W'$ there is a cone point. Thus every finite $W'$ possesses a cone point
and hence a branch point.

\begin{cor}\label{cor: cone} Let $x, x' \in W$ and let $y$ be a branch point for $\{x, x'\}$.
Then $(x|x')_y \in \{0, \frac{1}{2}\}$, in particular $|xx'| \geq |xy| + |yx'| - 1$.
\end{cor}

\begin{prop} For any metric space $\mathcal{RH}(Z)$ and $\mathcal{H}(Z)$ are roughly isometric. In particular, they are quasi-isometric.
\end{prop}

\begin{proof} Let $F\co \mathcal{RH}(Z) \to \mathcal{H}(Z)$ be such that $F(x_k)$ is some $v\in V_k$ such that $p_k^{-1}(x_k)\in B(v)$ and for any edge $e=[x,x']$, $F(e)$ is a geodesic path $[F(x),F(x')]$.

Let $x,x'$ be two vertices in $\mathcal{RH}(Z)$ joined by a radial geodesic $\gamma=x_0,..., x_n$ with $x_0=x$ and $x_n=x'$. Suppose $k=l(x)$, $k'=l(x')$ and $k'=k-n$. Then, for every $i=1,n$, $d_X(x_{i-1},x_i)\leq r^i$ and, therefore, there is a radial edge $[F(x_{i-1}),F(x_i)]$. Thus, $|F(x)F(x')|=|xx'|$.

Suppose now two vertices $x,x'\in \mathcal{RH}(Z)$ not joined by a radial geodesic. Let $y$ be a branch point for $\{x,x'\}$. Then $F(y)$ is a cone point for $\{F(x),F(x')\}$. Let us see that $|F(y)w|\leq 2$ for some branch point $w$ for $\{F(x),F(x')\}$. Let $k=l(y)=l(F(y))$ and $k\leq k'=l(w)$. Since $w$ is a brach point for $\{F(x),F(x')\}$, then $B(F(x))$ and $B(F(x'))$ are contained in $B(w)$ which has diameter $2r^{k'}$. In particular, $d_X(x,x')\leq 4r^{k'}<r^{k'-1}$ and there exists a cone point for $\{x,x'\}$ in $\mathcal{RH}(Z)$ at level $k'-1$. Hence $k\geq k'-1$ and $l(w)-l(F(y))\leq 1$. Since $B(F(y))\cup B(w)\neq \emptyset$ it follows that $|F(y),w|\leq 2$ in $\mathcal{H}(X)$.

By Corollary \ref{cor: cone}, $|xy|+|yx'|-1 \leq |xx'|\leq |xy|+|yx'|$ and by the corresponding lemma for $\mathcal{H}(Z)$, $|F(x)w|+|wF(x')|-1 \leq |F(x)F(x')|\leq |F(x)w|+|wF(x')|$.  Thus, by triangle inequality, 
$$|F(x)F(y)|+|F(y)F(x')|-5 \leq |F(x)F(x')| \leq |F(x)F(y)|+|F(y)F(x')|+4.$$ Also, as we saw above, $|xy|=|F(x)F(y)|$ and $|yx'|=|F(y)F(x')|$. Therefore,
$$|xx'|-5 \leq |xy|+|yx'|-5  \leq |F(x)F(x')| \leq |xy|+|yx'|+4\leq |xx'|+5.$$ 

Every vertex $v$ in $\mathcal{H}(Z)$ is at distance at most $1$ from the image $F(x)$ for every $x\in B(v)$ and therefore, $F$ is a rough isometry (in particular, a quasi-isometry).
\end{proof}

\begin{cor} Two complete metric spaces $Z,Z'$ are PQ-symmetric if and only if $\mathcal{RH}(Z)$ and $\mathcal{RH}(Z')$ are quasi-isometric.
\end{cor}

\begin{cor} $\mathcal{RH}(Z)$ is $\delta$-hyperbolic. In particular, it has uniformly generated $H_1$.
\end{cor}

\begin{prop} Let $l\co X \to \br$ be the level function on $\mathcal{HR}(X)$. Then, $l^{-1}([k,k+r])$ is connected if and only if $l^{-1}(k)$ is connected.
\end{prop}

\begin{proof} Any vertex in $l^{-1}([k,k+r])$ is connected by radial edges to a vertex in $l^{-1}(k)$. If $l^{-1}(k)$ is connected, then $l^{-1}([k,k+r])$ is connected.

Suppose two vertices $x,x'\in l^{-1}(k)$ connected by a path in $l^{-1}([k,k+r])$. Then, there is a sequence of vertices $x=x_0,x_1,...,x_n=x'$ in $l^{-1}([k,k+r])$ such that $\{x_{i-1},x_i\}$ is an edge in $l^{-1}([k,k+r])$ for $i=1,n$. For every $x_i$ with $l(x_i)=j$ there is a vertex $y_i=p_k(p_j^{-1}(x_i))\in l^{-1}(k)$. Since $d_X(x_i,x_{i+1})\leq r^k$, then either $y_i=y_{i+1}$ or $\{y_i,y_{i+1}\}$ is an edge in $l^{-1}(k)$. Therefore, $l^{-1}(k)$ is connected.
\end{proof}

\begin{cor} The level function is metrically proper on the connected components if and only if there is a constant $D>0$ such that every connected component of $l^{-1}(k)$ has diameter at most $D$ for any $k\in \bz$.
\end{cor}

Note that the level function is trivially bornologous.

\begin{cor} Let $Z$ be a metric space, $X=\mathcal{RH}(Z)$ and $l\co X \to \br$ its level function. If there is a constant $D>0$ such that every connected component of $l^{-1}(k)$ has diameter at most $D$ for any $k\in \bz$, then $Z$ is PQ-symmetric to an ultrametric space. Moreover, $X$ holds all the conditions in \ref{Th: application}.
\end{cor}

\begin{nota} Let $Z$ be a metric space, $X=\mathcal{RH}(Z)$ and $l\co X \to \br$ its level function. The identity induces a bijection between the vertices of any connected component of $l^{-1}(k)$ and the corresponding $r^k$-connected component on $Z$.
\end{nota}


\begin{cor} Let $Z$ be a metric space. If there are constants $D>0$ and $0<r<\frac{1}{6}$ such that every $r^k$-connected component is $D$-finitely $r^k$-connected for any $k\in \bz$, then $Z$ is PQ-symmetric to an ultrametric space.
\end{cor}

\section{Acknowledgments}

The author would like to express his
gratitude to the mathematics department of the University of Illinois at Chicago for their hospitality 
and to Kevin Whyte for his help and support producing this work.

The author was partially supported by MTM-2009-07030.

\end{document}